\documentclass{amsart}
\pdfoutput=1
\usepackage{graphics}
\usepackage{amsthm}
\usepackage{amssymb}
\usepackage{amscd}
\usepackage{amsmath}
\usepackage[all]{xy}
\usepackage[colorlinks]{hyperref}
\usepackage{color} %
\usepackage{booktabs}
\usepackage{mathtools}


\DeclareMathOperator{\Aut}{Aut}
\DeclareMathOperator{\Jac}{Jac}
\DeclareMathOperator{\ord}{ord}
\DeclareMathOperator{\PGL}{PGL}
\DeclareMathOperator{\Proj}{Proj}

\newcommand{\bbF}{{\operatorname{\mathbb F}}}
\newcommand{\bbP}{{\operatorname{\bf P}}}
\newcommand{\bbZ}{{\operatorname{\bf Z}}}

\newcommand{\calH}{{\mathcal H}}
\newcommand{\calL}{{\mathcal L}}
\newcommand{\calS}{{\mathcal S}}
\newcommand{\calV}{{\mathcal V}}

\renewcommand{\tilde}{\widetilde}

\makeatletter

\def\@marginparreset{\marginparstyle}
\def\marginparstyle{\SMALL\normalfont\raggedright\openup-2pt }
\long\def \@savemarbox #1#2{%
 \global\setbox #1%
     \vtop{%
       \hsize\marginparwidth
       \@parboxrestore
       \reset@font
       \@setnobreak
       \@setminipage
       \@marginparreset
       #2%
       \par
       \global\@minipagefalse
       }%
}

\DeclareRobustCommand\marginparhere[2][0pt]{%
 \ifhmode\unskip\fi
 \ifmmode\ssty\mathclose{\fi
   \rlap{\hskip\marginparsep\smash{%
     \vtop to 0pt{\marginparstyle \hsize\marginparwidth 
       \leftskip=#1 \rightskip=-#1 plus20pt \noindent#2\vss}}}%
 \ifmmode}\fi
}

\DeclareRobustCommand{\redden}{\@ifnextchar*
 {\@latex@error{{redden*} is only an environment}}
 {\@ifnextchar[{\r@dden}{\r@dd@n}}}
\def\r@dden[#1]{\ifmmode\@mperr{math}\else\ifinner\@mperr{inner}%
 \else\leavevmode\marginpar{\leavevmode#1\endgraf}\fi\fi\r@dd@n}
\def\r@dd@n{\def\reserved@a{redden}
 \ifx\@currenvir\reserved@a\redd@n\bgroup\ignorespaces
 \else\expandafter\redd@n\fi}
\def\endredden{\unskip\egroup}
\def\@mperr#1{\@latex@warning{redden in #1 mode: no marginpar possible}}
\def\redd@n#1{\leavevmode{\color{red}#1}}

\makeatother

\newtheorem{theorem}{Theorem}[section]
\newtheorem{lemma}[theorem]{Lemma}
\newtheorem{proposition}[theorem]{Proposition}
\newtheorem{corollary}[theorem]{Corollary}

\theoremstyle{definition}
\newtheorem{algorithm}[theorem]{Algorithm}

\theoremstyle{remark}
\newtheorem{remark}[theorem]{Remark}

\numberwithin{equation}{section}

\makeatletter
\@namedef{subjclassname@2020}{%
  \textup{2020} Mathematics Subject Classification}
\makeatother

\begin{document}

\title[Superspecial Howe curves of genus $4$]{Algorithms to enumerate\\ superspecial Howe curves of genus \texorpdfstring{$4$}{4}}

\author[Kudo]{Momonari Kudo}
\address{Department of Mathematical Informatics, Graduate School of Information Science and Technology,
The University of Tokyo, 7-3-1, Hongo, Bunkyo-ku,
Tokyo 113-8656, Japan}
\curraddr{}
\email{kudo@mist.i.u-tokyo.ac.jp}
\urladdr{\url{https://sites.google.com/view/m-kudo-official-website/english?authuser=0}}

\author[Harashita]{Shushi Harashita}
\address{Graduate School of Environment and Information Sciences,
Yokohama National University, 79-7, Tokiwadai, Hodogaya-ku,
Yokohama 240-8501, Japan}
\curraddr{}
\email{harasita@ynu.ac.jp}

\author[Howe]{Everett W. Howe}
\address{Unaffiliated mathematician, San Diego, CA 92104}
\curraddr{}
\email{however@alumni.caltech.edu}
\urladdr{\url{http://ewhowe.com}}



\date{29 July 2020}

\dedicatory{}

\begin{abstract}
A \emph{Howe curve} is a curve of genus $4$ obtained as the fiber product
of two genus-$1$ double covers of $\bbP^1$.
In this paper, we present a simple algorithm for testing isomorphism of Howe curves,
and we propose two main algorithms for finding and enumerating these curves:
One involves solving multivariate systems coming from Cartier--Manin matrices, while the other uses Richelot isogenies
of curves of genus $2$.
Comparing the two algorithms by implementation and by complexity analyses, we conclude that the latter enumerates curves more efficiently.
Using these algorithms, we show that there exist superspecial curves of genus $4$ in characteristic $p$ for every prime $p$ with $7 < p < 20000$.
\end{abstract}

\maketitle

\section{Introduction}\label{sec:intro}
\subsection{Background and motivation}
Let $K$ be an algebraically closed field of characteristic $p>0$.
A nonsingular curve over $K$ is called \emph{superspecial} (resp.\ \emph{supersingular}) if its Jacobian variety is isomorphic (resp.\ isogenous) to a product of supersingular elliptic curves.
Superspecial curves are not only theoretically interesting in algebraic geometry and number theory but also have many applications in coding theory, cryptology, and so on, because they tend to have many rational points and their Jacobian varieties have large endomorphism rings.
However, it is not always easy to find such curves,
and there are only finitely many superspecial curves for a given genus and characteristic.
One method of constructing
superspecial curves is to consider fiber products of superspecial curves of lower genera.
In this paper, we demonstrate that this method can be efficient by considering the simplest 
example in which the genus is at least $4$:
the case of Howe curves. 
A \emph{Howe curve}\footnote{So named by Senda and the first two authors in \cite{KHS20}.} is a curve of genus $4$ obtained as the fiber product of two genus-$1$ double covers $E_1\to\bbP^1$ and $E_2\to\bbP^1$.
In \cite{Howe}, the third author studied these curves in order to quickly construct genus-$4$ curves with many rational points.

\subsection{Related works}
The reason that we consider the case of genus $g\ge 4$ is that the enumeration of the isomorphism classes of superspecial curves with $g\le 3$ has already been done, by Deuring~\cite{Deuring} for $g=1$, by Ibukiyama, Katsura, and Oort \cite{IKO} for $g=2$, and by Brock~\cite{Brock} for $g=3$; see also Ibukiyama~\cite{Ibukiyama} and Oort~\cite{Oort91} for the existence of such curves for $g=3$.
In contrast to the case $g\le 3$, the existence or non-existence of a superspecial curve of genus $4$ in general characteristic is an open problem, although some results for specific small $p$ are known; 
see \cite[Theorem~1.1]{Ekedahl} for the non-existence for $p\le 3$ and \cite[Theorem~B]{KH17} for the non-existence for $p=7$.
As for enumeration, computational approaches have been proposed recently in \cite{KH18}, \cite{KH17}, and \cite{KH19} in the case of genus~$4$. The main strategy common to 
these papers is to parametrize a family of curves (canonical curves in 
the first two papers, hyperelliptic curves in the third), and then to
find the superspecial curves $X$ in these families
by computing the zeros of a multivariate system derived from the condition that the Cartier--Manin matrix of $X$ is zero.
With computer algebra techniques such as Gr\"{o}bner bases, the authors 
of these papers enumerated superspecial canonical
curves for $p \leq 11$ in \cite{KH18} and \cite{KH17} and superspecial hyperelliptic curves for $p \leq 23$ in~\cite{KH19}.
However, results for larger $p$ have not been obtained yet due to the cost of solving multivariate systems, and no complexity analysis is given in \cite{KH18}, \cite{KH17}, or~\cite{KH19}.

Now we  turn our attention to Howe curves.
Recently, it was proven in \cite{KHS20} that there exists a supersingular Howe curve in every positive characteristic.
In particular, the authors of \cite{KHS20} reduce the existence of such a curve
to the existence of a zero of a certain multivariate system, as follows:
They study a family of Howe curves realized as
$E_1\colon z^2=f_1(x)$ and $E_2\colon w^2=f_2(x)$ for cubic polynomials $f_1$ and $f_2$
parametrized by elements $(\lambda : \mu : \nu)$ of $\bbP^2$.
Let $C$ be the genus-$2$ curve $y^2 = f_1f_2$.
The supersingularity of $H$ is equivalent to
that of $E_1$, $E_2$ and $C$,
because there exists an isogeny of 
$2$-power degree from the Jacobian $J (H)$ to $E_1 \times E_2 \times J(C)$~\cite[Theorem~2.1]{Howe}).
Thus, once supersingular isomorphism classes of $E_1$ and $E_2$ are given,
finding supersingular curves $H$ is reduced to finding values of the parameter $(\lambda : \mu : \nu)$ that satisfy a multivariate system
derived from the supersingularity of $C$.
The authors of \cite{KHS20} deduced the existence of such a zero $(\lambda : \mu : \nu)$ from various algebraic properties of the defining polynomials of the system.

The above reduction is applicable also for the superspecial case, but the method used in \cite{KHS20} to prove the existence of solutions does not carry over well.
For this reason, the superspecial case is still open, and we are left to ask:
For which primes $p>7$ are there superspecial Howe curves in characteristic~$p$?

\subsection{Our contribution}
We study the existence of superspecial Howe curves by creating efficient
algorithms to produce and enumerate them. The following theorems summarize
some of what we have found.

\begin{theorem}\label{thm:main1}
For every prime $p$ with $7 < p < 20000$ or with $p\equiv 5\bmod 6$, there exists a superspecial Howe curve in characteristic~$p$.
\end{theorem}

\begin{theorem}\label{thm:main2}
For every prime $p$ with $7 < p \le 199$, the number of isomorphism classes of superspecial Howe curves in characteristic $p$ is given in 
Table~\textup{\ref{table:enumerate}}.
\end{theorem}

\begin{table}[ht]
\caption{For each prime $p$ from $11$ to $199$, we give the number $n(p)$ of superspecial Howe curves over $\overline{\bbF}_p$ and the ratio of $n(p)$ 
to the heuristic prediction $p^3/1152$ (see Section~\ref{sec:second}).}
\label{table:enumerate}
\vspace{-1ex}
\begin{center}
\renewcommand{\arraystretch}{0.9}
\begin{tabular}{rrrrrrrrrrr}
\toprule
  $p$ &  $n(p)$ &\small    Ratio &\hbox to 1em{}&   $p$ &  $n(p)$ &    Ratio &\hbox to 1em{}&   $p$ &  $n(p)$ &    Ratio \\
\cmidrule{1-3}\cmidrule{5-7}\cmidrule{9-11}
 $11$ &     $4$ &  $3.462$ &&  $67$ &   $260$ &  $0.996$ && $137$ &  $2430$ &  $1.089$ \\
 $13$ &     $3$ &  $1.573$ &&  $71$ &   $742$ &  $2.388$ && $139$ &  $2447$ &  $1.050$ \\
 $17$ &    $10$ &  $2.345$ &&  $73$ &   $316$ &  $0.936$ && $149$ &  $3082$ &  $1.073$ \\
 $19$ &     $4$ &  $0.672$ &&  $79$ &   $595$ &  $1.390$ && $151$ &  $3553$ &  $1.189$ \\
 $23$ &    $33$ &  $3.125$ &&  $83$ &   $655$ &  $1.320$ && $157$ &  $3427$ &  $1.020$ \\
 $29$ &    $45$ &  $2.126$ &&  $89$ &   $863$ &  $1.410$ && $163$ &  $3518$ &  $0.936$ \\
 $31$ &    $59$ &  $2.281$ &&  $97$ &   $802$ &  $1.012$ && $167$ &  $6268$ &  $1.550$ \\
 $37$ &    $41$ &  $0.932$ && $101$ &  $1207$ &  $1.350$ && $173$ &  $4780$ &  $1.064$ \\
 $41$ &   $105$ &  $1.755$ && $103$ &  $1151$ &  $1.213$ && $179$ &  $5771$ &  $1.159$ \\
 $43$ &    $79$ &  $1.145$ && $107$ &  $1237$ &  $1.163$ && $181$ &  $5419$ &  $1.053$ \\
 $47$ &   $235$ &  $2.608$ && $109$ &  $1193$ &  $1.061$ && $191$ &  $9610$ &  $1.589$ \\
 $53$ &   $167$ &  $1.292$ && $113$ &  $1323$ &  $1.056$ && $193$ &  $6298$ &  $1.009$ \\
 $59$ &   $259$ &  $1.453$ && $127$ &  $2013$ &  $1.132$ && $197$ &  $6839$ &  $1.030$ \\
 $61$ &   $243$ &  $1.233$ && $131$ &  $2606$ &  $1.335$ && $199$ &  $8351$ &  $1.221$ \\
\bottomrule
\end{tabular}
\vskip 3ex
\end{center}
\end{table}

The upper bounds on $p$ in these two theorems can easily be increased. For example, on a 2.8 GHz Quad-Core Intel Core i7 with 16GB RAM, computing the 
$8351$ superspecial Howe curves in characteristic $199$ using method (B) below
took $124$ seconds in Magma. Finding examples of superspecial Howe curves for
every $p$ between $7$ and $20000$ took $680$ minutes on the same machine.

In this paper we discuss two strategies, (A) and (B) below, to find superspecial Howe curves. We also show how isomorphisms between Howe curves can be easily detected from the data that defines them (C).

\subsection*{(A) \texorpdfstring{$(E_1,E_2)$}{(E1,E2)}-first, using Cartier--Manin matrices.}
In this strategy, we use the same realization of Howe curves as in \cite{KHS}, that is,
the fiber product of
$E_1\colon z^2y = x^3 + A_1 \mu^2 xy^2 + B_1 \mu^3y^3$ and
$E_2\colon w^2y = (x-\lambda)^3 + A_2 \mu^2 (x-\lambda)y^2 + B_2 \mu^3y^3$
over $\bbP^1=\Proj K[x,y]$.
We enumerate pairs $(E_1,E_2)$ of supersingular elliptic curves
so that $C$ is superspecial.
We first discuss the field of definition of superspecial Howe curves (cf.\;Proposition \ref{prop:field_of_def}), which enables us to reduce the size of our search space drastically.
Specifically, the coordinates $A_1$, $B_1$, $A_2$, $B_2$, $\lambda$, $\mu$, $\nu$ belong to $\bbF_{p^2}$, whereas in the supersingular case~\cite{KHS20} these coordinates can generate larger 
subfields of $\overline{\bbF}_{p}$.
For the test of superspeciality, we use the criterion that the Cartier--Manin matrix of $C$ must be zero \cite[Lemma~1.1(i)]{IKO}. 
This reduces the enumeration problem to solving a system of algebraic equations. See Section \ref{sec:first} for the details of this strategy, including a complexity analysis.

\subsection*{(B) \texorpdfstring{$C$}{C}-first, using Richelot isogenies.}
The second strategy first enumerates
superspecial curves $C\colon y^2=f(x)$ of genus $2$ with $f(x)$ of degree $6$
and then enumerates decompositions $f(x)=f_1(x)f_2(x)$ with $f_i(x)$ of degree $3$
so that there is a $b$ that makes both curves $E_i\colon y^2 = (x-b)f_i(x)$ supersingular.
The moduli space of curves of genus $2$ is of dimension $3$.
As this dimension is bigger than the space of $(\lambda:\mu:\nu)\in \bbP^2$ considered in (A), this strategy, a priori, looks inefficient. But, surprisingly, we conclude that
strategy (B) enumerates superspecial Howe curves much more efficiently than does (A).
The advantage of (B) comes from making use of Richelot isogenies.
Specifically, we construct some superspecial curves of genus $2$ by 
gluing supersingular elliptic curves together along their $2$-torsion
\cite[\S3]{HLP}, and then produce more such curves by applying Richelot
isogenies to the curves already produced. This procedure terminates because there
are only finitely many superspecial curves of genus $2$, and a recent result
of Jordan and Zaytman~\cite[Corollary~18]{JZ} shows that we obtain all isomorphism classes of superspecial curves of genus $2$ in this way.

\subsection*{(C) A new isomorphism test for Howe curves.}
Strategy (A) above produces many 
not-necessarily-distinct Howe curves, so to prevent overcounting we are left with the task of
producing a unique representative for each isomorphism class. As every Howe curve is canonical (cf.\;Proposition \ref{HoweIsCanonical}), one may check whether two Howe curves are isomorphic by using the isomorphism test for canonical curves given in \cite[\S6.1]{KH17}, whose implementation is found in \cite[\S4.3]{KH18}. This turns out to be very costly, because it uses many 
Gr\"obner basis computations. Our Corollary~\ref{C:isomorphism} gives a 
much simpler isomorphism test, based on the observation that a Howe curve
is completely determined (up to isomorphism) by the degree-$2$ map to a genus-$2$ curve it is provided with by virtue of its definition as a fiber product. This isomorphism test is added on as a separate step in strategy (A), but is baked into the algorithm we use for strategy (B).

\subsection*{Acknowledgments}
The first and the second authors thank Everett Howe for
joining as the third author; he told them the second strategy (B), which had not been considered in the earlier version \cite{KH20}.
The third author thanks Professors Kudo and Harashita for inviting him to join them in this work.
The authors thank the referees for their careful reading and for helpful suggestions and comments.
This work was supported by JSPS Grant-in-Aid for Scientific
Research (C) 17K05196, JSPS Grant-in-Aid for Research Activity Start-up 18H05836 
and 19K21026, and JSPS Grant-in-Aid for Young Scientists 20K14301.

\section{Howe curves and their superspeciality}\label{sec:Howe}
In this section, 
we recall the definition of Howe curves, show that they are canonical,
and give a computational criterion for their superspeciality.

Let $K$ be an algebraically closed field of characteristic $p\neq 2$.
A \emph{Howe curve} over $K$ is a curve which is isomorphic to the desingularization of
the fiber product $E_1 \times_{{\bbP}^1} E_2$ of two genus-$1$ double covers
$E_i\to {\bbP}^1$ ramified over $S_i$, where
each $S_i$ consists of 4 points and where $|S_1\cap S_2|=1$.

Given a Howe curve, there is an automorphism of $\bbP^1$ that takes the common ramification point
of the two genus-$1$ double covers to infinity. Then the curves $E_i$ can be written
$w^2 = f_1$ and $z^2 = f_2$ for separable monic cubic polynomials $f_i\in K[x]$ that
are coprime to one another, where $x$ generates the function field of~$\bbP^1$.

\begin{lemma}\label{HoweIsCanonical}
Every Howe curve is a canonical curve of genus $4$.
\end{lemma}
\begin{proof}
Let $H$ be a Howe curve, normalized as above so that it is given as the fiber product
of $w^2 = f_1$ and $z^2 = f_2$ for coprime separable monic cubic polynomials $f_1$ and~$f_2$.
For each $i$ let $f_i^{(\rm h)}\coloneqq y^3 f_i(x/y)\in K[x,y]$ be the homogenous cubic obtained
from $f_i$, and let $H'$ be the curve defined in $\bbP^3=\Proj K[x,y,z,w]$ by
\begin{align*}
z^2-w^2 &= q(x,y)\\ 
z^2y &= f_1^{(\rm h)}(x,y), 
\end{align*}
where $q(x,y)$ is the quadratic form
\[
q(x,y) = (f_1^{(\rm h)}(x,y)-f_2^{(\rm h)}(x,y))/y.
\]
Note that $H'$ and $E_1\times_{\bbP^1} E_2$ are isomorphic if
the locus $y=0$ is excluded.
It is straightforward to see that $H'$ is nonsingular,
since $f_1$ and $f_2$ are separable and are coprime.
Hence $H$ and $H'$ are isomorphic to one another (cf.\,\cite[Proposition~II.2.1]{Sil}).

It is well-known that any nonsingular curve defined by
a quadratic form and a cubic form in $\bbP^3$
is a canonical curve of genus $4$ \cite[Example~IV.5.2.2]{Har}.
\end{proof}

To study the superspeciality of Howe curves, we first look at the decomposition of
their Jacobians. Let $f_1$ and $f_2$ be coprime separable monic cubic polynomials, as above.
Let $f = f_1 f_2$ and consider the hyperelliptic curve $C$ of genus $2$ defined by
$u^2 = f.$
By~\cite[Theorem~2.1]{Howe}, there exist two isogenies
\begin{align*}
\varphi&\colon  J(H) \longrightarrow E_1 \times E_2 \times J(C)\\
\psi&\colon E_1 \times E_2 \times J(C) \longrightarrow J(H)
\end{align*}
such that $\varphi\circ\psi$ and $\psi\circ\varphi$ are both multiplication by $2$.

Suppose now that the characteristic $p$ of $K$ is an odd prime.
Then 
$\psi\circ\varphi$ is an automorphism of the $p$-kernel of $J(H)$
and
$\varphi\circ\psi$ is an automorphism of the $p$-kernel of $E_1 \times E_2 \times J(C)$,
so $J(H)[p]$ and $E_1[p] \times E_2[p] \times J(C)[p]$ are isomorphic. 
Hence $H$ is superspecial if and only if $E_1$ and $E_2$ are supersingular and $C$ is superspecial. 

Now we recall a criterion for the superspeciality of $C$.
Let $\gamma_i$ be the coefficient of $x^i$ in $f^{(p-1)/2}$, and set
\[
a=\gamma_{p-1},\quad
b=\gamma_{2p-1},\quad
c=\gamma_{p-2}\quad \text{and}\quad
d=\gamma_{2p-2}.
\]
Let $M$ be the matrix
\begin{equation}\label{Cartier-Manin matrix}
M \coloneqq \begin{pmatrix}
a^p & c^p \\
b^p & d^p
\end{pmatrix}.
\end{equation}
Then $M$ is a Cartier--Manin matrix for $C$, that is, there is a basis for 
$H^0(C,\Omega_C^1)$ so that left multiplication by $M$ represents the (semi-linear)
action of the Cartier operator; here $\Omega_C^1$ is the sheaf of differential $1$-forms on $C$.
(For information about Cartier--Manin matrices, see~\cite{AH}, which addresses issues with 
earlier literature, including the standard reference~\cite[\S 2]{Yui}.)

\begin{lemma}\label{lem:ssp}
Let $H$ be a Howe curve as above.
Then $H$ is superspecial if and only if $E_1$ and $E_2$ are supersingular and $a=b=c=d=0$.
\end{lemma}

\begin{proof}
We already noted that $H$ is superspecial if and only if $E_1$ and $E_2$
are supersingular and $C$ is superspecial. But $C$ is 
superspecial if and only if the Cartier operator acts trivially on $H^0(C,\Omega_C^1)$
\cite[Theorem~4.1]{Nygaard}.
\end{proof}

\section{Detecting isomorphisms of Howe curves}\label{sec:IsomTest}
In this section, we give an efficient criterion for determining
whether two Howe curves are isomorphic or not.
This criterion will be used in both the first and the second approach
to enumerating superspecial Howe curves over a finite field.

We continue to work over an algebraically closed field of characteristic $p\neq 2$.
Recall from Section \ref{sec:Howe} that a Howe curve is the desingularization of
the fiber product of two genus-$1$ double covers of $\bbP^1$, where the 
ramification loci of the two covers overlap in exactly one point. This means 
that a Howe curve is precisely a genus-$4$ curve $H$ that fits into a $V_4$-diagram
of the following form, where $C$ is a curve of genus~$2$ and $E_1$ and $E_2$ are curves of genus $1$:
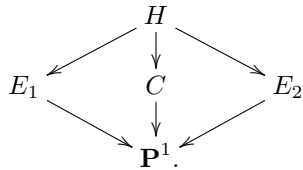
\[
\xymatrix@=3ex{
&&H\ar[lld]\ar[rrd]\ar[d]&&\\
E_1\ar[rrd] & &C\ar[d] & &E_2\ar[lld]\\
&&\bbP^1\rlap{.}&&
}
\]

If $E_1\to\bbP^1$ ramifies at points $P$, $Q_1$, $Q_2$, and $Q_3$, and 
if $E_2\to\bbP^1$ ramifies at $P$, $R_1$, $R_2$, and $R_3$, then the Weierstrass
points of $C$ are the points lying over $Q_1$, $Q_2$, $Q_3$, $R_1$, $R_2$, and $R_3$. 
On the other hand, the point $P$ splits in the cover $C\to\bbP^1$,
and we let $P_1$ and $P_2$ be the points of $C$ lying over~$P$.

Thus, to specify a Howe curve, it is enough to provide three pieces of information:
\begin{enumerate}
\item A genus-$2$ curve $C$.
\item An unordered pair of disjoint sets $\{W_1, W_2\}$, each consisting of
      three Weierstrass points of $C$.
\item An unordered pair of distinct points $\{P_1, P_2\}$ on $C$ that are mapped to
      one another by the hyperelliptic involution.
\end{enumerate}      

These three things determine the $V_4$-diagram above, and hence also determine
the double cover $\eta\colon H\to C$, which we call the \emph{structure map} for the given data.
Of course, if $\alpha$ is an automorphism of $C$ then $\{\alpha(W_1), \alpha(W_2)\}$ and 
$\{\alpha(P_1), \alpha(P_2)\}$ will give us a double cover $H\to C$ that is isomorphic to $\eta$,
namely $\alpha \eta$.

\begin{lemma}
\label{L:data}
The data specifying a Howe curve is recoverable \textup(up to automorphisms of $C$\textup)
just from the structure map $\eta\colon H\to C$.
\end{lemma}

\begin{proof}
The map $C\to\bbP^1$ is unique (up to automorphism of $\bbP^1$), so we recover the entire map 
$H\to C\to \bbP^1$ from $\eta$. This map is a Galois extension with group~$V_4$, 
so we recover the genus-$1$ curves in the extension, and hence the division of
the Weierstrass points of $C$. The pair of points $\{P_1, P_2\}$ is simply the set
of ramification points of $\eta$.
\end{proof}

\begin{theorem}
Two structure maps $\eta_1\colon H \to C_1$ and $\eta_2\colon H\to C_2$ starting from the 
same Howe curve $H$ are isomorphic to one another. That is, there is an isomorphism 
$\gamma\colon C_1\to C_2$ and an automorphism $\delta\colon H\to H$ such that the
following diagram commutes\textup{:}
\[
\xymatrix@C=4ex{
H\ar[rr]^\delta\ar[d]_(0.4){\eta_1}&&H\ar[d]^(0.4){\eta_2}\\
C_1\ar[rr]^\gamma && C_2\rlap{.}
}
\]
\end{theorem}

\begin{proof}
Let $U_1$ and $U_2$ be the $V_4$-subgroups of $\Aut H$ specified by $\eta_1$ and $\eta_2$,
and let $S$ be the $2$-Sylow subgroup of $\Aut H$ that contains $U_1$.
By conjugating $U_2$ by an automorphism $\delta$ (and thereby replacing $\eta_2$
with $\eta_2 \delta$) we may assume that $U_2$ is also contained in~$S$. 
Let $\alpha_1$ and $\alpha_2$ be the involutions of $H$ corresponding to the
double covers $\eta_1$ and $\eta_2$, and for each $i$ let $\beta_i$ and $\gamma_i$
be the other nonzero elements of $U_i$.

If $\alpha_1$ and $\alpha_2$ are conjugate to one another in $S$ (or even in $\Aut H$), we are done. 
So assume, to get a contradiction, that $\alpha_1$ and $\alpha_2$ lie in different conjugacy classes of~$S$.

We know that the quotient of $H$ by the subgroup $\langle\alpha_i\rangle$ has genus~$2$,
while the quotients of $H$ by $\langle\beta_i\rangle$ and by $\langle\gamma_i\rangle$ have
genus~$1$. The same is true for all of the conjugates of $\alpha_i$, $\beta_i$, and 
$\gamma_i$ in $S$. Moreover, if we have any $V_4$-subgroup of $S$, giving rise to
a diagram
\begin{equation}
\begin{split}
\label{EQ:V4}
\xymatrix@=3ex{
&&H\ar[lld]\ar[rrd]\ar[d]&&\\
Y_1\ar[rrd] & &Y_2\ar[d] & &Y_3\ar[lld]\\
&&X\rlap{,}&&
}
\end{split}
\end{equation}
then none of the curves $Y_i$ can have genus~$0$ (by Lemma~\ref{HoweIsCanonical}),
and so the only possibilities are that (a) all of the $Y_i$ have genus~$2$ and $X$ has genus~$1$,
or (b) one of the $Y_i$ has genus~$2$, the other two have genus~$1$, and $X$ has genus~$0$.
(This follows from the fact that in any diagram such as~\eqref{EQ:V4}, the genus of $H$
is the sum of the genera of the $Y_i$ minus twice the genus of $X$; this follows, for
instance, from~\cite[Theorem~B]{KR}.)
Thus, given two commuting involutions in $S$, if we know the genera of the quotients
of $H$ they produce, we can deduce the genus of the quotient of $H$ by their product.

Our strategy, then, will be to enumerate all possible $2$-groups $S$ that occur as the
$2$-Sylow subgroup of the automorphism group of a non-hyperelliptic curve~$H$ of genus~$4$,
along with all possible pairs $U_1$ and $U_2$ of $V_4$-subgroups of $S$ that contain
elements $\alpha_1$ and $\alpha_2$ that are not conjugate in~$S$. We will assume that 
$\alpha_1$ and $\alpha_2$ generate genus-$2$ curves, while the other involutions in $U_1$ and $U_2$
generate genus-$1$ curves. Given these assumptions, we deduce, for as many involutions as we can, the
genera of the curves associated to these involutions. 

Suppose $\delta$ is an involution in $S$ for which we know that the quotient ${Y\coloneqq H/\langle\delta\rangle}$
has genus~$2$. Let $T$ be the centralizer of $\delta$ in $S$.
Then the quotient $T/\langle\delta\rangle$ is contained in the automorphism group of the genus-$2$ curve $Y$. 
Using Igusa's classification of the automorphism groups of genus-$2$ curves~\cite[\S8]{Igusa},
we can show that there are only eight $2$-groups that appear as subgroups of the automorphism
groups of genus-$2$ curves. If $T/\langle\delta\rangle$ is not one of these groups, then we have shown that the values
of $U_1$, $U_2$, $\alpha_1$, and $\alpha_2$ cannot correspond to two different realizations
of $H$ as a Howe curve.

In order to use this strategy, we need a good bound on the sizes of automorphism groups
of non-hyperelliptic curves of genus $4$ in characteristic not~$2$.
A result of Henn~\cite[Satz~1]{Henn} (see also~\cite{GK}) shows
that in characteristic $p>2$, the order of the automorphism group of a curve 
of genus~$g$ is strictly less than $8g^3$, except possibly when the curve is of one of
the following types:
\begin{enumerate}
\item $x^n + y^m = 1$, where $n = 1 + p^a$ for some $a>0$ and $m\mid n$, or
\item $y^p - y = x^n$, where $n = 1 + p^a$ for some $a>0$.
\end{enumerate}
The first type of curve has genus $(n-2)(m-1)/2$, and if this is equal to $4$
then either we have $n=10$ and $m=2$ (and $p=3$) or we have $n=6$ and $m=3$ (and $p=5$). 
In the first case the curve is hyperelliptic; in the second case, as Henn
notes, the automorphism group has order $360$, which is less than $8g^3$.
The second type of curve has genus $p^a (p-2)/2$, which is never equal to~$4$.
Thus, it will suffice for us to look at every $2$-group $S$ of order less than $8\cdot 4^3 = 512$.

We implemented this computation in Magma; the code is available on the third author's web site.
We ran our code on all $2$-groups of order less than $512$, and the only group not eliminated was 
$S\cong(\bbZ/2\bbZ)^3.$

For this $S$, our computation shows that of the seven involutions in $S$, three
give genus-$2$ quotients and four give genus-$1$ quotients, and the three elements
that give genus-$2$ quotients sum to zero. Now consider the seven $V_4$-subgroups $T$ of $S$.
Each such $T$ gives us a diagram like~\eqref{EQ:V4} above. 
For the $T$ that contains the three genus-$2$ involutions, the genus of $H/T$ is~$1$,
while for the other six $V_4$-subgroups $T$, the genus of $H/T$ is~$0$.

Let us consider the diagram of subextensions between $H$ and its quotient $H/S\cong \bbP^1$.
We label the elements of $S$ by vectors in $\bbF_2^3$, and we label the $V_4$-subgroups in the
same way, with the convention that a $V_4$-subgroup labeled by $v$ contains the elements with labels $g$
such that the dot product of $v$ and $g$ is~$0$. Then the diagram of subextensions,
with their genera, is as follows:
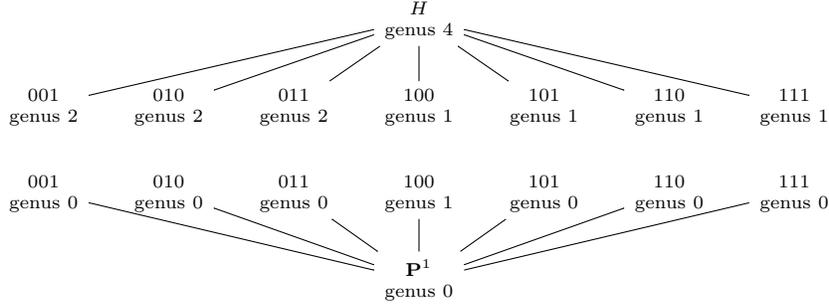
\[
\xymatrix@=3ex{
&&&\genfrac{}{}{0pt}{}{H}{\textup{genus 4}}\ar@{-}[llld]\ar@{-}[lld]\ar@{-}[ld]\ar@{-}[d]\ar@{-}[rd]\ar@{-}[rrd]\ar@{-}[rrrd]\\
\genfrac{}{}{0pt}{}{001}{\textup{genus 2}}&
\genfrac{}{}{0pt}{}{010}{\textup{genus 2}}&
\genfrac{}{}{0pt}{}{011}{\textup{genus 2}}&
\genfrac{}{}{0pt}{}{100}{\textup{genus 1}}&
\genfrac{}{}{0pt}{}{101}{\textup{genus 1}}&
\genfrac{}{}{0pt}{}{110}{\textup{genus 1}}&
\genfrac{}{}{0pt}{}{111}{\textup{genus 1}}\\
\genfrac{}{}{0pt}{}{001}{\textup{genus 0}}&
\genfrac{}{}{0pt}{}{010}{\textup{genus 0}}&
\genfrac{}{}{0pt}{}{011}{\textup{genus 0}}&
\genfrac{}{}{0pt}{}{100}{\textup{genus 1}}&
\genfrac{}{}{0pt}{}{101}{\textup{genus 0}}&
\genfrac{}{}{0pt}{}{110}{\textup{genus 0}}&
\genfrac{}{}{0pt}{}{111}{\textup{genus 0}}\\
&&&\genfrac{}{}{0pt}{}{\bbP^1}{\textup{genus 0}}\ar@{-}[lllu]\ar@{-}[llu]\ar@{-}[lu]\ar@{-}[u]\ar@{-}[ru]\ar@{-}[rru]\ar@{-}[rrru]&&&\\
}
\]
(For visual clarity, we have left off the heads of the arrows, and omitted the $21$ arrows between the middle layers.)
But this configuration of genera is not possible; consider for example the following subdiagram:
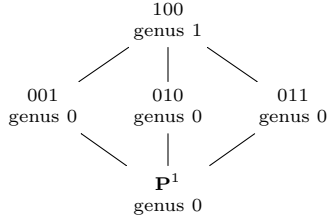
\[
\xymatrix@=3ex{
&\genfrac{}{}{0pt}{}{100}{\textup{genus 1}}\ar@{-}[ld]\ar@{-}[d]\ar@{-}[rd]&\\
\genfrac{}{}{0pt}{}{001}{\textup{genus 0}}&
\genfrac{}{}{0pt}{}{010}{\textup{genus 0}}&
\genfrac{}{}{0pt}{}{011}{\textup{genus 0}}\\
&\genfrac{}{}{0pt}{}{\bbP^1}{\textup{genus 0}}\ar@{-}[lu]\ar@{-}[u]\ar@{-}[ru]&\\
}
\]
This diagram violates the genus property we mentioned below diagram~\eqref{EQ:V4}.

This contradiction shows that the involutions $\alpha_1$ and $\alpha_2$ corresponding
to the structure maps $\eta_1$ and $\eta_2$ lie in the same conjugacy class of $\Aut H$, so
that $\eta_1 = \eta_2\delta$ for an automorphism $\delta$ of $H$.
\end{proof}

\begin{corollary}
\label{C:isomorphism}
Two triples $(C, \{W_1, W_2\}, \{P_1, P_2\})$
and $(C', \{W_1', W_2'\}, \{P_1', P_2'\})$ give isomorphic Howe curves
if and only if there is an isomorphism $C \to C'$ that takes 
$\{W_1, W_2\}$ to $\{W_1', W_2'\}$ and $\{P_1, P_2\}$ to $\{P_1', P_2'\}$. \qed
\end{corollary}
This isomorphism test is very fast; it simply requires determining whether there are any automorphisms of $\bbP^1$ that
respect the sets of Weierstrass points and their divisions, and that take the $x$-coordinate of $P_1$ and $P_2$ to that of $P_1'$ and $P_2'$.

\section{First approach: Reduction to solving multivariate systems}\label{sec:first}
In this section and the next,
we present two approaches to solving the problem of enumerating superspecial Howe curves.
As we mentioned in Section~\ref{sec:intro},
the first approach, described in this section, enumerates pairs of supersingular elliptic curves $E_1\colon w^2 = f_1$ and $E_2\colon z^2 = f_2$ such that $C\colon y^2=f_1f_2$ is superspecial.
For this, we shall apply a construction of Howe curves given in \cite{KHS20}.
While this construction is different from the original one of \cite{Howe}, it can easily reduce our problem to finding roots of polynomial systems.

\subsection{Reduction to solving multivariate systems over finite fields}\label{subsec:red}
Let $K$ be an algebraically closed field in characteristic $p > 3$.
In \cite{KHS20}, the authors parametrize the space of all Howe curves by the projective plane $\bbP^2$.
We here briefly recall the parametrization; see \cite[\S2]{KHS20} for more details.
Let $y^2 = x^3 + A_i x + B_i$ ($i=1,2$) be two (nonsingular) elliptic curves, where $A_1,B_1,A_2,B_2\in K$.
Let $\lambda, \mu, \nu$ be elements of $K$ such that (i) $\mu \ne 0$ and $\nu \neq 0$, and (ii) $f_1$ and $f_2$ are coprime, where 
\begin{align}
f_1(x) &= x^3 + A_1 \mu^2 x + B_1\mu^3,\label{f1}\\
f_2(x) &= (x-\lambda)^3 + A_2 \nu^2(x-\lambda) + B_2 \nu^3.\label{f2}
\end{align}
A point $(\lambda : \mu : \nu) \in \bbP^2 (K)$ satisfying (i) and (ii) is said to be \emph{of Howe type} in \cite{KHS20}.
Note that the isomorphism classes of $E_1$ and $E_2$ are independent of the choice of $(\lambda,\mu,\nu)$ provided $\mu\ne 0$ and $\nu\ne 0$.
Then the desingularization $H$ of the fiber product $E_1\times_{{\bbP}^1}E_2$ is a Howe curve, and vice versa.

This parametrization, together with the criterion of superspeciality in Section~\ref{sec:Howe}, enables us to reduce the search for superspecial Howe curves into solving multivariate systems over $K$; it suffices to compute the solutions $(\lambda : \mu : \nu) \in \bbP^2(K)$ (of Howe type) to $a = b = c = d = 0$, where $a$, $b$, $c$ and $d$ are the entries of the Cartier--Manin matrix of the hyperelliptic curve $C\colon y^2 = f_1 f_2$.
Note that $a$, $b$, $c$ and $d$ are homogeneous as polynomials in $\lambda$, $\mu$ and $\nu$, and that $\ord_*(-)=O(p)$ for $*=\lambda,\mu,\nu$ and for $-=a,b,c,d$.

Note that the multivariate systems above are zero-dimensional, since
there are only finitely many points $(\lambda : \mu : \nu)$ parameterizing supersingular Howe curves (cf.\;\cite{KHS20}), whence the same thing holds for superspecial cases.
In fact, we may assume that the coordinates $A_1$, $B_1$, $A_2$, $B_2$, $\lambda$, $\mu$ and $\nu$ belong to $\bbF_{p^2}$:

\begin{proposition}\label{prop:field_of_def}
Any superspecial Howe curve is $K$-isomorphic to $H$ obtained as above for $A_1$, $B_1$, $A_2$, $B_2$, $\mu$, $\nu$ and $\lambda$ belonging to $\bbF_{p^2}$.
\end{proposition}
\begin{proof}
It suffices to consider the case of $K=\overline{\bbF}_{p^2}$,
since every supersingular elliptic curve can be defined over $\bbF_{p^2}$
and $(\lambda,\mu,\nu)$ is a solution of $a=b=c=d=0$.
Let $H'$ be a superspecial Howe curve over $K=\overline{\bbF}_{p^2}$.
Choose $E'_1$ and $E'_2$ over $K$ so that $H'$ is the normalization of
$E'_1\times_{\bbP^1} E'_2$.
It is well-known that $H'$ descends to a curve $H$ over $\bbF_{p^2}$
such that the Frobenius map $F$ (the $p^2$-power map) on $\Jac(H)$ is $p$ or $-p$ and all automorphisms of $H$ are defined over $\bbF_{p^2}$ (cf.\;the proof of \cite[Theorem~1.1]{Ekedahl}). Let $E_1$ and $E_2$ be the quotients of $H$ corresponding to $E'_1$ and $E'_2$. The quotient $E_i$ of $H$ is obtained by an involution $\iota_i \in\Aut(H)$, and therefore is defined over $\bbF_{p^2}$.
The quotient of $H$ by the group generated by $\iota_1$ and $\iota_2$
is isomorphic to $\bbP^1$ over ${\bbF_{p^2}}$.
Let $S_i$ be the set of the ramified points of $E_i \to \bbP^1$.
Since $S_1\cap S_2$ consists of a single point, this point is invariant under the action of the absolute Galois group of $\bbF_{p^2}$ and therefore is an $\bbF_{p^2}$-rational point. An element of $\PGL_2(\bbF_{p^2})$ sends this point to the infinite point of~$\bbP^1$.
Since the Frobenius map $F$ on $E_i$ is also $\pm p$,
the other elements $P$ of $S_i$ (which are $2$-torsion points on $E_i$)
are also $\bbF_{p^2}$-rational by $F(P)=\pm p P = P$.
This implies the desired result.
\end{proof}

\subsection{Concrete algorithm}\label{subsec:alg1}
Based on the reduction described in the previous subsection, we present a concrete algorithm (Algorithm \ref{alg:main1} below).
\begin{algorithm}\label{alg:main1}
\ 
\begin{description}
\item[Input] A rational prime $p > 3$.
\item[Output] A list $\calH(p)$ of superspecial Howe curves, each of which is represented by a pair $(f_1, f_2)$ of polynomials $f_1, f_2 \in \bbF_{p^2}[x]$.
\end{description}
\begin{enumerate}
\item Compute the set $\calS(p)$ of representatives of the $\overline{\bbF}_{p}$-isomorphism classes of supersingular elliptic curves in characteristic $p$ so that each representative is given in Weierstrass form $E_{A,B}\colon y^2 = f_{A,B}(x) \coloneqq x^3 + A x + B$ by a pair $(A, B)$ of elements in $\bbF_{p^2}$.
\item Set $\calH_0(p)$ $\leftarrow$ $\emptyset$.
For each pair of $E_{A_1,B_1}$ and $E_{A_2, B_2}$ in $\calS(p)$, possibly choosing $(A_1,B_1) = (A_2,B_2)$, conduct the following (2-1)-(2-3) to compute all $(\lambda , \mu, \nu ) \in (\bbF_{p^2})^3$ of Howe type such that the desingularization $H$ of $E_1 \times_{\bbP^1} E_2$ is superspecial, where $E_1\colon w^2 = f_1$ (resp.\ $E_2\colon z^2 = f_2$) is an elliptic curve $\bbF_{p^2}$-isomorphic to $E_{A_1,B_1}$ (resp.\ $E_{A_2,B_2}$).
\begin{enumerate}
\item[(2-1)] Compute the Cartier--Manin matrix $M$ given in \eqref{Cartier-Manin matrix}.
\item[(2-2)] Compute the set $\calV(A_1,B_1,A_2,B_2)$ of elements $(\lambda, \mu, \nu) \in (\bbF_{p^2})^3$ (with $\nu = 1$) such that $M=0$.
\item[(2-3)] For each $(\lambda, \mu, \nu) \in \calV(A_1,B_1,A_2,B_2)$:
If $\mu \neq 0$ and $\nu \neq 0$, set $\calH_0(p)$ $\leftarrow$ $\calH_0(p) \cup \{ (f_1, f_2) \}$, where $f_1$ and $f_2$ are as in \eqref{f1} and \eqref{f2}.
\end{enumerate}
{\bf Note:} By Lemma 4.4 and  Proposition 4.6 of \cite{KHS20}, for each root $(\lambda, \mu , \nu)$ computed in Step (2-2), the cubics $f_1$ and $f_2$ are coprime if $\mu \neq 0$ and $\nu \neq 0$.
Moreover, it suffices to compute elements $(\lambda, \mu, \nu)$ with $\nu = 1$; see Remark 4.2 of \cite{KH20} for more details.
\item Set $\calH(p)$ $\leftarrow$ $\emptyset$.
For each $(f_1, f_2) \in \calH_0(p)$:
If the Howe curve $H$ represented by $(f_1, f_2)$ is not isomorphic to any Howe curve of $\calH(p)$, set $\calH(p)$ $\leftarrow$ $\calH(p) \cup \{ H \}$.
\end{enumerate}
\end{algorithm}

The complexity of this algorithm is estimated as $\tilde{O}(p^6)$, as long as $\# \calH_0 (p) = O(p^3)$; see Subsection \ref{subsec:alg1comp} below for more details.

\begin{remark}
If one would search for a single example of a superspecial Howe curve (or determine the non-existence of such a curve), it suffices to decide the (non)-existence of a root in Step (2-2).
In this case, it will be estimated in the next subsection that the complexity is $\tilde{O} (p^5)$. 
\end{remark}

\subsection{Complexity of the first approach}\label{subsec:alg1comp}

We here briefly discuss the complexity of Algorithm \ref{alg:main1} together with several variants of computing the roots of a multivariate system in Step (2-2).
For reasons of space, we give only a summary of the estimation of the complexity, and refer to \cite[\S5.1]{KH20} for most of details.
In the following, all time complexity bounds refer to arithmetic complexity, which is the number of operations in $\bbF_{p^2}$.
We denote by $\mathsf{M}(n)$ the time to multiply two univariate polynomials over $\bbF_{p^2}$ of degree $n$.

For Step (1), one can check that its complexity is dominated by the cost of computing all supersingular $j$-invariants in characteristic $p$.
This cost is bounded by $O ({\log}^{2} (p) \mathsf{M}(p) ) = \tilde{O}(p)$, see \cite[\S5.1.1]{KH20} for details.

For Step (2), clearly the complexities of Steps (2-1) and (2-2) are larger than that of Step (2-3).
In Step (2-1), we compute the Cartier--Manin matrix $M$ from $f\coloneqq f_1 f_2$ \emph{with indeterminates} $\lambda$ and $\mu.$
The cost of computing $M$ is bounded by $\tilde{O}(p^3)$, see Remark \ref{rem:compHW} below.
In Step (2-2), there are three variants (i)-(iii) to compute all $(\lambda, \mu, \nu) \in (\bbF_{p^2})^3$ with $\nu = 1$ such that $M=0$, where $M$ is the Cartier--Manin matrix as in \eqref{Cartier-Manin matrix} with entries $a$, $b$, $c$ and $d$.
\begin{enumerate}
\item[(i)] Use brute force to enumerate all    $(\lambda, \mu) \in (\bbF_{p^2})^2$ to check $M =0$ or not.
\item[(ii)] Regard one of $\lambda$ and $\mu$ as a variable.
For simplicity, regard $\lambda$ as a variable.
For each $\mu \in \bbF_{p^2}$, compute the roots in $\bbF_{p^2}$ of $G\coloneqq \gcd(a, b, c, d) \in \bbF_{p^2}[\lambda ]$.
\item[(iii)] Regarding both $\lambda$ and $\mu$ as variables, use an approach based on resultants.
\end{enumerate}

It is estimated that the complexity of (i) is $O(p^5)$, and that those of (ii) and (iii) are bounded by the same bound $\tilde{O} (p^4)$;
more precisely the upper-bound of the complexity of (ii) is less than that of (iii) if we consider logarithmic factors, see \cite[\S5.1.2]{KH20}.

From this, we adopt the fastest variant (ii) with complexity $\tilde{O}(p^4)$ in our implementation.
The number of $(\lambda, \mu, \nu)$ with $\nu = 1$ computed in Step (2-2) is $ \leq p^2 \times \mathrm{deg}(G)  = O (p^3)$.
Since the number of possible choices of $(E_{A_1,B_1},E_{A_2,B_2})$ is $\# \calS(p) = O (p^2)$, computing $(\lambda, \mu, \nu)$ with $\nu = 1$ for all $(E_{A_1,B_1},E_{A_2,B_2})$ is done in $\# \calS(p) \times  \tilde{O}(p^4) = \tilde{O} (p^6) $ operations in $\bbF_{p^2}$.

For Step (3), the complexity of this step heavily depends on the number of superspecial Howe curves obtained in Step (2), that is, $\# \calH_0(p)$.
Since each isomorphism test is done in $O(1)$, the complexity of Step (3) is $O ( (\# \calH_0(p) )^2 )$.
As of this writing, we have not succeeded in finding any sharp bound on $\# \calH_0(p)$.
We can naively estimate $\# \calH_0(p) = O (p^5)$ from the complexity analysis of Step (2), whereas we expect $\# \calH_0 (p) = O(p^3)$ from the practical behavior \cite[\S4.2, Table 1]{KH20}.
Thus, the complexity of Step (3) is naively $O(p^{10})$, but in practice $O (p^{6})$ which does not exceed the complexity of Steps (1)-(2).

Note that to determine the (non-)existence of a superspecial Howe curve, it is not necessary to compute a root in Step (2-2), but it suffices to compute the gcd $G$ only.
Since each gcd can be computed in time $\tilde{O}(p)$ by fast gcd algorithms, one can verify that the total complexity of this variant of Algorithm \ref{alg:main1} is $\tilde O(p^5)$.

\begin{remark}\label{rem:compHW}
In Step (2-1), we compute a Cartier--Manin matrix over $\bbF_{p^2}[\lambda, \mu]$.
Bostan, Gaudry, and Schost showed that
in general, computing the Cartier--Manin matrix $M$ of a hyperelliptic curve $y^2 = f(x)$ defined over a field $K$ can be accomplished by multiplying matrices obtained from recurrences 
for the coefficients of $f(x)^n$; see \cite[\S8]{BGS} or \cite[\S2]{HS2} for details.
The algorithm of Harvey and Sutherland \cite{HS2}, which is an improvement of their earlier algorithm~\cite{HS} presented at ANTS XI, is also based on this reduction, and it is the fastest algorithm to compute $M$ for the case of $K=\bbF_p$.
From this, we suspect that one of the best ways to compute $M$ in Step (2-1) would be to extend the Harvey--Sutherland algorithm \cite{HS2} to the case of $\bbF_{p^2} (\lambda, \mu)$.
However, since we have not yet succeeded in making this extension, we compute $M$ using the reduction mentioned above, or by using formul{\ae} given in \cite[\S4]{KHS20} for $M$ specific to Howe curves.
It is estimated (to appear in a revised version of \cite{KH20}) that the complexity of the latter method is bounded by $\tilde{O}(p^3)$, which is less than or equal to that of Step (2-2).
\end{remark}

\section{Second approach: Use of Richelot isogenies of 
genus-\texorpdfstring{$2$}{2} curves}
\label{sec:second}
In this section we propose another approach to enumerating superspecial Howe 
curves. As opposed to the approach in Section \ref{sec:first}, this second 
approach \emph{starts} with a superspecial genus-$2$ curve $C$, and then looks
to see whether it will fit into a $V_4$-diagram with supersingular elliptic 
curves. While this is precisely the structure of Algorithm~5.7 of \cite{Howe},
the problem remains: How can we \emph{quickly} produce a list of \emph{all} of
the superspecial genus-$2$ curves? We begin by addressing this question.

\subsection{Computing superspecial curves of genus \texorpdfstring{$2$}{2}}
To produce a list  $\calL$ of all superspecial genus-$2$ curves, we use a
variant of~\cite[Algorithm~5.7]{Howe}.
Each superspecial genus-$2$ curve has a unique model defined over $\bbF_{p^2}$ that is maximal over $\bbF_{p^2}$.
Given one such curve, all of the curves that are Richelot isogenous to it are also maximal superspecial curves. Thus, given a not-necessarily-complete list of maximal superspecial curves, we can add curves to the list as follows: 
We go through the list one curve at a time.
For each $C$ we compute the curves that are Richelot isogenous to it, and we add each such curve to the list if it is not already on it.
To seed our list, we can use the curves that are $(2,2)$-isogenous to a product of maximal elliptic curves.
Then a result of Jordan and Zaytman~\cite[Corollary~18]{JZ} shows that this 
procedure will generate a complete list $\calL$ of all superspecial genus-$2$
curves.

The exact number of curves on the list $\calL$ is given by a result of
Ibukiyama, Katsura, and Oort~\cite[Theorem~3.3]{IKO}. The exact answer depends
on the congruence class of $p$ modulo $120$, but it follows from their
result that for $p>3$ we have
\[\#\calL = \frac{(p-1)(p^2 + 25 p + 166)}{2800} + c,
\text{\quad where\quad}
\frac{-1}{16}\le c\le \frac{209}{180}.
\]

\subsection{Testing whether a genus-\texorpdfstring{$2$}{2} curve fits into a \texorpdfstring{$V_4$}{V4}-diagram}

For each $C \in \calL$, given by an equation
\[   
y^2 = (x - a_1)(x - a_2)(x - a_3)(x - a_4)(x - a_5)(x - a_6),
\]
we would like to try to fit $C$ into a Howe curve diagram. 
For each of the ten ways of splitting the Weierstrass points into two groups of three (for example, into $\{ \{a_1,a_2,a_3\}, \{a_4,a_5,a_6\} \}$), we could then ask for the values of $b$ such that the two genus-$1$ curves
\begin{equation}\label{eq:genus1:1}
y^2 = (x - b)(x - a_1)(x - a_2)(x - a_3)
\end{equation}
and
\begin{equation}\label{eq:genus1:2}
y^2 = (x - b)(x - a_4)(x - a_5)(x - a_6)
\end{equation}
are both supersingular.
(We also consider ``$b=\infty$,'' corresponding to the
curves $y^2 = (x - a_1)(x - a_2)(x - a_3)$ and $y^2 = (x - a_4)(x - a_5)(x - a_6)$.
Since there are about $p/12$ supersingular $j$-invariants and hence about $p/2$ supersingular $\lambda$-invariants, there are about $p/2$ values of $b$ that will make the first curve \eqref{eq:genus1:1} supersingular, and we can compute these values in time $\tilde{O}(p)$.
For each $b$, we then check whether the second curve \eqref{eq:genus1:2} is supersingular.
If we were to model this as choosing a random $\lambda$-invariant
in $\bbF_{p^2}$ and asking whether it is supersingular, 
we would expect success with probability around $1/(2p)$.

It is easy to incorporate isomorphism testing into this algorithm so that it produces each superspecial Howe curve exactly once: 
All we have to do is keep track of how the automorphism group of $C$
acts on the divisions of its Weierstrass points and on the good values of~$b$.

Thus, in time $\tilde{O}(p^4)$, we can produce unique representatives for each superspecial Howe curve. 
Heuristically, the number of superspecial Howe curves we find should be the number of superspecial genus-$2$ curves ($\approx p^3/2880$), times the number of Weierstrass point divisions ($10$), times the number of values of $b$ that make the first  elliptic curve supersingular ($\approx p/2$), times the probability that the second curves is supersingular ($\approx 1/(2p)$). Heuristically, then, we
expect to find about $p^3/1152$ superspecial Howe curves.


\subsection{Concrete algorithm}
\begin{algorithm}\label{alg:main2}
\ 
\begin{description}
\item[Input] A rational prime $p > 3$.
\item[Output] A list $\calH(p)$ of superspecial Howe curves, 
      each of which is represented by a pair $(f_1, f_2)$ of
      polynomials $f_1, f_2\in \bbF_{p^2}[x]$, corresponding to
      the curve $y^2 = f_1, z^2 = f^2$.
\end{description}
\begin{enumerate}
\item Compute the set $\mathrm{MaxEll}(p^2)$ of $\bbF_{p^2}$-isomorphism 
      classes of $\bbF_{p^2}$-maximal elliptic curves over 
      $\bbF_{p^2}$. Since every supersingular curve has a unique maximal
      twist, this can be done as in Step (1) of Algorithm \ref{alg:main1}.
\item $\calL \leftarrow \emptyset$.
      For each pair $(E, E')$ of elements in $\mathrm{MaxEll}(p^2)$,
      compute the (at most $6$) curves $C$
      whose Jacobians are $(2,2)$-isogenous to $E\times E'$ 
      (see~\cite[\S 3]{HLP}). Adjoin each of these to $\calL$ if it is
      not isomorphic to an element of $\calL.$
\item Write $\calL = \{C_1,\ldots,C_n\}$. Set $i=1$.
      \begin{enumerate}
      \item \label{previousstep}
            For each nonsingular curve $C'$ which is Richelot isogenous to
            $C_i$: If $C'$ is not isomorphic to any element of $\calL$,
            set $N\leftarrow |\calL|$ and put $C_{N+1}\coloneqq C'$ and
            $\calL\leftarrow \calL\cup\{C_{N+1}\}$.
      \item If $i < |\calL|$, set $i\leftarrow i+1$ and go to 
            \eqref{previousstep}.
      \end{enumerate}
\item \label{nextstep}
      Set $\calH(p)\leftarrow \emptyset$. 
\item \label{bigstep}
      For each $C \in \calL$: Check whether $C$ fits into a Howe curve
      diagram with supersingular double covers $E_i \rightarrow \bbP^1$. 
      \begin{enumerate}
      \item For each splitting of the Weierstrass point of $C$ into two 
            disjoint three-element sets, compute the $j$-invariants of the
            genus-$1$ curves \eqref{eq:genus1:1} and \eqref{eq:genus1:2}, 
            as functions of the indeterminate~$b$. 
            Find the values of $b$ that make the first curve supersingular,
            and for each such value, check to see whether the second curve 
            is supersingular.
            Record each value of $b$ for which both curves are supersingular.
      \item Using Corollary~\ref{C:isomorphism}, find unique representatives
            $y^2 = f_1, z^2 = f_2$ for the curves produced in the previous 
            step, and adjoin $(f_1,f_2)$ to $\calH(p)$.
      \end{enumerate}
\end{enumerate}      
\end{algorithm}

We noted in the previous subsection that Step~\eqref{bigstep} takes
$\tilde{O}(p^4)$ arithmetic operations over $\bbF_{p^2}$, and the
other steps clearly take fewer operations than this.

\section{Implementations and proofs}\label{sec:algorithm}

In this section, we describe our implementations of the algorithms
in the previous sections and our proofs of the main results stated in the Introduction. As we have seen, there are two approaches to enumerating
superspecial Howe curves: (A) $(E_1,E_2)$-first
and (B) $C$-first. 
The arguments in the previous sections show that (B) has an advantage in the complexity analysis.
Here we see that (B) is much superior to (A) also when we execute their implementations. Indeed, Theorems \ref{thm:main1} and \ref{thm:main2} in 
the Introduction were obtained by Magma implementations based on (B) that were run on a PC with ubuntu 16.04 LTS OS at 3.40GHz CPU (Intel Core i7-6700) and 15.6 GB memory. The same result for $p\le 53$ was obtained by implementing the method (A) over Magma with an execution by the same PC. Although
it took 11871 seconds to obtain Theorem \ref{thm:main2} for $p\le 53$ by (A), the second strategy (B) finishes the enumeration for $p \le 199$
only in 924 seconds; see Table \ref{tableAB:2} for benchmark timing data for small $p$.
\begin{table}[ht]
\caption{Benchmark timing data for (A) Algorithm \ref{alg:main1}
and (B) Algorithm \ref{alg:main2}. All times shown are in seconds.}
\label{tableAB:2}
\vspace{-1ex}
\begin{center}
\renewcommand{\arraystretch}{0.9}
\begin{tabular}{rrrrrrrrrrr}
\toprule
  $p$ & \multicolumn{1}{c}{(A)} & \multicolumn{1}{c}{(B)} &\hbox to 1em{}& $p$ &  \multicolumn{1}{c}{(A)} &  \multicolumn{1}{c}{(B)} &\hbox to 1em{}& $p$ &  \multicolumn{1}{c}{(A)} &  \multicolumn{1}{c}{(B)}\\
\cmidrule{1-3}\cmidrule{5-7}\cmidrule{9-11}
  $5$ &  $0.02$ &  $0.08$ &&  $19$ &   $6.14$ &  $0.12$ && $41$ &  $1118.63$ &  $0.71$ \\
  $7$ &  $0.01$ &  $0.01$ &&  $23$ &  $27.59$ &  $0.21$ && $43$ &  $1423.26$ &  $0.80$ \\
 $11$ &  $0.17$ &  $0.04$ &&  $29$ & $114.70$ &  $0.31$ && $47$ &  $2686.17$ &  $1.03$ \\
 $13$ &  $0.76$ &  $0.05$ &&  $31$ & $193.82$ &  $0.34$ && $53$ &  $5678.32$ &  $1.46$ \\
 $17$ &  $3.92$ &  $0.09$ &&  $37$ & $617.23$ &  $0.54$ &&      &            &         \\
\bottomrule
\end{tabular}
\end{center}
\end{table}
The code for our implementations is available on the first author's web site
for (A) and on the third author's web site for (B).
In case (A), it is very costly to find Cartier--Manin matrices, and in addition to that there are many pairs $(E_1, E_2)$ of supersingular elliptic curves. This fact is consistent with the complexity analysis in Section \ref{subsec:alg1comp}.
On the other hand, the method (B) contains few intensive computations
and it enables us to find and enumerate superspecial Howe curves very efficiently. 

The preceding remarks prove the computational results in Theorems~\ref{thm:main1} and~\ref{thm:main2}, and we are left to prove
the statement in Theorem~\ref{thm:main1} concerning primes $p\equiv 5\bmod 6.$
This fact is shown by using the Howe curve defined by $E_1\colon z^2y=x^3+y^3$ and $E_2\colon w^2y=x^3+ay^3$ with $a\in\{-1,1/4\}$.
Indeed, if $p\equiv 5 \bmod 6$, then
these two elliptic curves are supersingular and moreover
$y^2 = (x^3+1)(x^3+a)$ is superspecial.
This can be checked 
by observing that the curve has two nonhyperelliptic involutions,
given by $(x,y)\mapsto (a^{1/3}/x, \pm a^{1/2}y/x^3)$, so that its Jacobian
is $(2,2)$-isogenous to a product of elliptic curves. For $a = -1$ we find that
these two curves are both isomorphic to the $j=0$ curve with CM by $-3$, 
and for $a = 1/4$ we find that they are both isomorphic to the $j = -12288000$
curve with CM by $-27$. In both cases, these elliptic curves are supersingular 
for primes $p\equiv 5\bmod 6$.

We remark that this curve for $a=1/4$ is isomorphic to the curve $X^3+Y^3+W^3=2YW+Z^2=0$ in $\bbP^3$ studied by the first author in \cite{K19}, by the correspondence $x=X$, $y=Y+W$, $z=\sqrt{-3/2}Z$ and $w=\sqrt{-3/4}(Y-W)$.

\bibliography{ANTS}
\bibliographystyle{hplaindoi}

\end{document}